\documentclass[10pt]{amsart}
\usepackage[utf8]{inputenc}

\usepackage{amsmath,amssymb,verbatim,amsthm,wasysym}
\usepackage{comment}
\usepackage{fp}
\usepackage{rotating}
\usepackage{accents} 
\usepackage{longtable}
\usepackage{mathtools}
\usepackage{centernot}

\DeclarePairedDelimiter\floor{\lfloor}{\rfloor}
\usepackage{tikz}
\usetikzlibrary{shapes}

\numberwithin{equation}{section}
\newtheorem{theorem}{Theorem}[section]
\newtheorem{lemma}[theorem]{Lemma}

\newtheorem{corollary}[theorem]{Corollary}

\theoremstyle{definition}

\newtheorem{remark}[theorem]{Remark}

\newcommand{\vertexdef}{\tikzstyle{vertex}=[draw,circle,fill=white,minimum size=1mm,inner sep=0pt]
			\tikzstyle{cv}=[white,draw,circle,fill=blue,minimum size = 1.4mm, inner sep=0pt]}

\newcommand{\PIVAL}{3.14159265358979323846264338} 
\newcounter{i} 
\newcommand{\Circulant}[2] { 
	\begin{tikzpicture}
	\vertexdef 
	\setcounter{i}{0}
	\whiledo{\value{i}<#1}{ 
		\FPmul\tempA{2}{\thei} 
		\FPdiv\tempB{\PIVAL}{#1} 
		\FPmul\tempC{\tempA}{\tempB} 
		\FPcos\varX{\tempC} 
		\FPsin\varY{\tempC} 
		\stepcounter{i} 
		\FPround\varX{\varX}{3}
		\FPround\varY{\varY}{3}
		\node (\thei) at (\varX,\varY)[place]{ }; 
	    \foreach \x in {#2} { 
			\pgfmathparse{mod(\x+\thei,#1)} 
			\let\tempB\pgfmathresult
			\pgfmathparse{mod(\thei-\x,#1)} 
			\let\tempA\pgfmathresult
			\ifthenelse{\lengthtest{\tempA pt < 1 pt}}{\FPadd\tempA{\tempA}{#1}}{}
			\ifthenelse{\lengthtest{\tempB pt < 1 pt}}{\FPadd\tempB{\tempB}{#1}}{}
			\ifthenelse{\lengthtest{\tempA pt > \thei pt}}{}{\ifthenelse{\thei = \tempA}{}{\draw [] (\thei) to (\tempA)}};
			\ifthenelse{\lengthtest{\tempB pt > \thei pt}}{}{\ifthenelse{\thei = \tempB}{}{\draw [] (\thei) to (\tempB)}};
		}
	}
	\setcounter{i}{0}
	\whiledo{\value{i}<#1}{ 
		\FPmul\tempA{2}{\thei} 
		\FPdiv\tempB{\PIVAL}{#1} 
		\FPmul\tempC{\tempA}{\tempB} 
		\FPcos\varX{\tempC} 
		\FPsin\varY{\tempC} 
		\stepcounter{i} 
		\FPround\varX{\varX}{3}
		\FPround\varY{\varY}{3}
		\node[vertex] at (\varX,\varY){};
	}
	\end{tikzpicture}
}

\newcommand{\mob}{\mathbin{\Bowtie}} 
\newcommand{\sq}{\mathbin{\square}}  
\newcommand{\R}{\ensuremath{\mathbb{R}}}

\begin{document}
\tikzstyle{place}=[draw,circle,minimum size=1.2mm,inner sep=1pt,outer sep=-1.1pt,fill=black]

\title{Maximum nullity and zero forcing  of circulant graphs}

\author[L. Duong]{Linh Duong}
\address{Department of Mathematics,  University of St. Thomas,  St. Paul, MN, 55105, USA}
\email{linhduong0527@gmail.com}

\author[B.K. Kroschel]{Brenda K. Kroschel}
\address{Department of Mathematics,  University of St. Thomas,  
St. Paul, MN, 55105, USA }
\email{bkkroschel@stthomas.edu}

\author[M. Riddell]{Michael Riddell}
\address{Department of Mathematics \& Statistics\\
McMaster University \\
Hamilton, ON, L8S 4L8, Canada}
\curraddr{386 Whitney Ave.
Hamilton, ON
L8S 2H4}
\email{riddelmj@mcmaster.ca, michaeljamesriddell@gmail.com}

\author[K.N. Vander Meulen]{Kevin N. Vander Meulen}
\address{Department of Mathematics\\
Redeemer University College, Ancaster, ON, L9K 1J4, Canada}
\email{kvanderm@redeemer.ca}

\author[A. Van Tuyl]{Adam Van Tuyl}
\address{Department of Mathematics \& Statistics\\
McMaster University \\
Hamilton, ON, L8S 4L8, Canada}
\email{vantuyl@math.mcmaster.ca}
\date{\today}



\begin{abstract}
It is well-known that the zero forcing number of a graph provides a lower bound on the minimum rank of a graph. 
In this paper we bound and characterize the zero forcing number of certain circulant graphs, including some bipartite circulants, cubic
circulants, and circulants which are torus products, 
to obtain bounds on the minimum
rank and the maximum nullity. We also evaluate when the zero forcing number will give equality.
\end{abstract}

\maketitle

\setlength{\parindent}{0pt}


\section{\textbf{Introduction}}
Let $G$ be a simple finite graph with vertex set $V(G)$ and
edge set $E(G)$.  Suppose in the graph $G$ some vertices are filled and some are unfilled.  The {\it filling rule} is as follows:  if a vertex $v \in V(G)$ is filled and has exactly one unfilled neighbor, $w$, then vertex $v$ forces $w$ to be filled, and $v$ is referred to as
a \emph{forcing vertex}.  
Given $F \subseteq V(G)$, the 
{\it final filling} of $F$ is the set of
filled vertices obtained by initially filling the vertices of $F$ and leaving every vertex in $V(G) \setminus F$ unfilled and applying the filling rule until no more vertices can be filled.  The set $F$ is called a {\it zero forcing set} if the final filling of $F$ is $V(G)$.  The 
terminology of zero forcing arose in the context of forcing
entries of a null vector to be zero 
as first described in \cite{AIM}.   An example of a zero forcing set 
is given in Figure \ref{figure1}.
 \begin{figure}[ht]
	\label{fig:forceEG}
	\begin{center}
		\begin{tikzpicture}
		\vertexdef
		\node (0) at (-0.5,0.5){$G_1$};
		\draw (0,0) -- (0,1);
		\draw (0,1) -- (1,1);
	    \draw (1,0) -- (1,1);
	    \draw (0,0) -- (1,0);
	    \draw (1,0) -- (2,0);
	    \draw (1,1) -- (2,1);
	    \draw (2,1) -- (2,0);
		\node[cv] (0,0) {};
		\node[vertex] at (1,0){};
		\node[cv] at (0,1) {};
		\node[vertex] at (1,1){};
		\node[vertex] at (2,0){};
		\node[vertex] at (2,1){};
	\end{tikzpicture} \qquad
	\begin{tikzpicture}
                 \vertexdef
		\node (0) at (2.5,0.5) {$G_2$};
		\draw (0,0) -- (0,1);
		\draw (1,0) -- (1,1);
	    \draw (0,0) -- (1,0);
	    \draw (0,1) -- (1,1);
	    \draw (1,0) -- (2,0);
	    \draw (1,1) -- (2,1);
	    \draw (2,1) -- (2,0);
	    \draw (0,0) -- (2,1);
		\node[cv] (0,0) {};
		\node[vertex] at (1,0){};
		\node[cv] at (0,1) {};
		\node[vertex] at (1,1){};
		\node[vertex] at (2,0){};
		\node[vertex] at (2,1){};
	\end{tikzpicture}
				\end{center}
	\caption{The filled vertices are a zero forcing set in $G_1$ but not in $G_2$.}\label{figure1}
\end{figure}
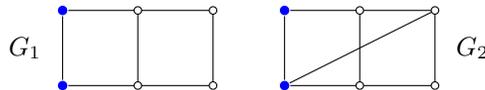

In various applications, it is of interest
to find the cardinality of a smallest zero forcing set 
in $G$ (which always exists since $V(G)$ is a trivial zero forcing set). 
The \emph{zero forcing number} of $G$, denoted $Z(G)$, is the minimum cardinality of a zero forcing set for a graph $G$. Determining $Z(G)$ is NP-hard \cite{AA} in general, but has been calculated for some well-known classes of graphs (see, for example, \cite{AIM, FH, Cat}). Variations of
zero forcing have been useful in communication complexity, quantum mechanics, graph theory, and
some inverse eigenvalue problems (see \cite{FH} for references).
The zero forcing number originated in \cite{AIM} as a technique
to find a bound on the minimum rank of a symmetric matrix associated with a graph.  Let $\mathcal{S}(G)$ denote the set of symmetric matrices over $\R$ whose
graph is $G$. In particular, if $G$ is a graph with vertices $v_0,v_1,..., v_{n-1}$, then $A \in \mathcal{S}(G)$ if $A$ is a real symmetric matrix
such that for $i\neq j$,  $A_{ij}\neq 0$ if and only if $v_i$ is adjacent to $v_j$ in $G$.  Note that if $A \in \mathcal{S}(G)$, then 
there is no restriction on the diagonal entries of $A$.
Let $M(G)=\max \{ {\rm{nullity}}(A)\ | \  A \in \mathcal{S}(G) \}$. It was demonstrated in
\cite{AIM} that $Z(G)$ provides an upperbound on $M(G)$.

\begin{theorem}\cite{AIM}\label{Mbound} Let $G$ be a graph and let $F \subset V(G)$ be a zero forcing set of $G$.
Then $M(G) \leq |F |$, and thus $M(G) \leq Z(G)$.
\end{theorem}

Recent work~\cite{AEG} describes families of graphs for which equality holds in Theorem~\ref{Mbound}, that is, families of graphs $G$ with $M(G)=Z(G)$. 
If we let ${\rm mr}(G)=\min \{ {\rm rank}(A)\ |\ A\in \mathcal{S}(G) \}$, then the rank theorem tells us that ${\rm mr}(G)+M(G)=n$. Hence Theorem~\ref{Mbound} demonstrates that zero forcing can provide a lower bound on the minimum rank of any symmetric matrix associated with a graph. 

Before going forward,
we state some known facts about zero forcing, 
which can be found in \cite{AIM}.

\begin{lemma}\cite{AIM} \leavevmode 
\begin{enumerate}\label{lem:facts}   
\item \label{lem:kregular}
 For any $k$-regular graph $G$, $Z(G)\geq k$. 
\item For $n>1$, $M(K_n)=Z(K_n)=n-1$.
\item For $n\geq 3$, $Z(C_n)=M(C_n)=2$.
\item For $n\geq 5$, $Z(\overline{C_n})=M(\overline{C_n})=n-3$, where
$\overline{G}$ denotes the complement of the graph $G$.
\item \label{union} For disjoint graphs $G$ and $H$, $Z(G \cup H)=Z(G)+Z(H).$
\item \label{lem:Cartesian} For the Cartesian product of graphs $G$ and $H$, 
$Z(G \sq  H) \leq \min \{Z(G)|H|, Z(H)|G| \} $. \item \label{lem:KsC} For $m\geq 2$, $M(K_2 \sq C_m)=Z(K_2 \sq C_m)=\min \{m,4\}.$ 
\end{enumerate}
\end{lemma}

In this paper we explore the zero forcing number for various classes of circulant graphs. Section \ref{properties} defines
circulant graphs and  reviews some of their properties.  In addition, we extend Deaett and Meyer's results on consecutive circulants \cite{DM}. The maximum nullity and zero forcing number of circulant graphs that are bipartite is explored in Section \ref{bipartite}.  For every bipartite circulant $G$ considered in this section, 
$Z(G) = M(G)$.  Section \ref{torus} introduces the torus product of a graph, noting that certain circulant graphs can be viewed as a torus product. The section explores the zero forcing number and maximum nullity for several cases of torus products.  The M\"obius ladder is a special case of a torus product.  We note that for many circulant graphs $G$ which are torus products, the numbers
$Z(G)$ and $M(G)$ are again equal, but there are still cases for which this is an open question.   In Section 5, we 
show that for all cubic circulant graphs $G$, 
$Z(G) = M(G)$,
and we compute
this value.


\section{\textbf{Properties of circulant graphs}}\label{properties}

We recall some of the properties of circulant graphs, and derive some basic results on the zero forcing number 
and minimal rank for this family.  For standard graph theory
terminology, see \cite{W}.

Given an integer $n\geq 1$ and a subset 
$S \subseteq \{1,2,\ldots, \floor{\frac{n}{2}}\}$, a \emph{circulant graph} $C_n(S)$ is a graph with vertex set $V(G) = \{v_0,v_1,\ldots,v_{n-1}\}$ and edge set $E(G) = \{\{v_i,v_{i+j}\} ~|~ i \in \{0,\ldots,n-1\}$ and $j \in S\}$, taking subscripts modulo $n$. Note that if 
$S = \{s_1,\ldots,s_t\}$, then we will abuse notation
and write $C_n(s_1,\ldots,s_t)$ instead of $C_n(\{s_1,\ldots,
s_t\})$.  Furthermore, we assume that
$s_1 < s_2 < \cdots < s_t$.
Some examples of circulant graphs are given in Figures 
\ref{figure2} and \ref{fig:12-124}.

\begin{figure}[ht]
	\label{fig:k2cart}
	\begin{center}
		\begin{tikzpicture}
		\vertexdef
		\node (n0) at (0,-0.5) {$v_0$};
		\node (n1) at (0,1.5)  {$v_7$};
		\node (n2) at (1,-0.5) {$v_2$};
		\node (n3) at (1,1.5) {$v_9$};
		\node (n4) at (2,-0.5)  {$v_4$};
		\node (n5) at (2,1.5) {$v_{11}$};
		\node (n6) at (3,-0.5)  {$v_6$};
		\node (n7) at (3,1.5) {$v_{13}$};
		\node (n8) at (4,-0.5) {$v_8$};
		\node (n9) at (4,1.5)  {$v_1$};
		\node (n10) at (5,-0.5) {$v_{10}$};
		\node (n11) at (5,1.5)  {$v_{3}$};
		 \node (n12) at (6,-0.5) {$v_{12}$};
		\node (n13) at (6,1.5) {$v_{5}$};
		[thick,scale=0.8]%

		\foreach \x in {0}
		\draw (0,0) -- (6,0);
		\draw (0,1) -- (6,1);
		
		\draw (0,0) to[out=10, in=170] (6,0);
		\draw (0,1) to[out=350, in=190] (6,1);
		
		\draw (0,0) -- (0,1);
		\draw (1,0) -- (1,1);
		\draw (2,0) -- (2,1);
		\draw (3,0) -- (3,1);
		\draw (4,0) -- (4,1);
		\draw (5,0) -- (5,1);
		\draw (6,0) -- (6,1);
		
		\node[cv] (0,0) {};
		\node[cv] at (1,0) {};
		\node[vertex] at (2,0) {};
		\node[vertex] at (3,0) {};
		\node[vertex] at (4,0) {};
		\node[vertex] at (5,0) {};
		\node[vertex] at (6,0)  {};
		\node[cv] at (0,1) {};
		\node[cv] at (1,1) {};
		\node[vertex] at  (2,1) {};
		\node[vertex] at (3,1) {};
		\node[vertex] at (4,1) {};
		\node[vertex] at (5,1) {};
		\node[vertex] at (6,1) {};

		\end{tikzpicture}\end{center}
	\caption{The graph $C_{14}(2,7) \cong   
	K_2 \sq C_7$ (see Section 4) with a zero forcing set.}\label{figure2}
\end{figure}
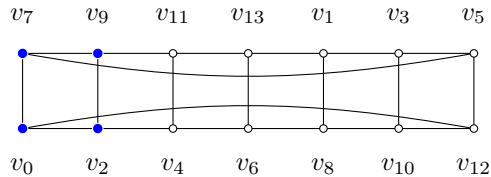

\begin{figure}[ht]
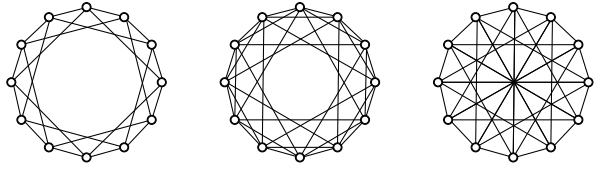

    \centering
    \[\Circulant{12}{1,3} \qquad \Circulant{12}{1,2,4}
    \qquad \Circulant{12}{1,4,6}\]
    \caption{The circulant graphs $C_{12}(1,3)$, $C_{12}(1,2,4)$
    and $C_{12}(1,4,6)$.}
    \label{fig:12-124}
\end{figure}

Since circulant graphs are vertex transitive, they are regular graphs.  In particular,  if 
$G = C_n(s_1,\ldots,s_t)$, then $G$ is $(2t-1)$-regular if $2s_t = n$, and $2t$-regular otherwise.  Combining this observation with Lemma~\ref{lem:facts}(1) gives a lower bound on  the zero forcing number of a circulant graph:

\begin{theorem}\label{thm:regular}
	Suppose $G = C_n(s_1,\ldots,s_t)$. 
	If $2s_t = n$, then $Z(G)\geq 2t-1$.
	If $2s_t \neq n$, then $Z(G)\geq 2t$.
\end{theorem}

Not every circulant graph is a connected graph (see Figure~\ref{fig:discon}). The connected circulant graphs were characterized by Boesch and Tindell \cite{BT}; in the
statement below, we write $gG$ to denote $g$ disjoint
copies of the graph $G$.

\begin{theorem}
\label{thm:discon}
If $G=C_n(s_1,s_2,\ldots,s_t),$ then $G$ is connected if and only if $gcd(n,s_1,s_2,\ldots,s_t)=1$. 
If $\gcd(s_1,s_2,\ldots,s_t,n)=g$, then  $C_n(s_1,s_2,\ldots,s_t) \cong gC_{\frac{n}{g}}(\frac{s_1}{g},\frac{s_2}{g},\ldots,\frac{s_t}{g})$.
\end{theorem}

\begin{figure}[ht]
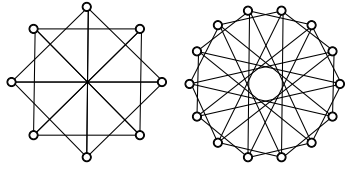

	\[\Circulant{8}{2,4}\ \ \Circulant{14}{2,6}\]
	\caption{The disconnected circulant graphs $C_{8}(2,4) \cong 2K_4 = K_4\cup K_4$, and $C_{14}(2,6) \cong 2C_7(1,3)$.}
	\label{fig:discon}
\end{figure}

Using  the function $f: \mathbb{Z}_n \rightarrow \mathbb{Z}_n$, defined by $f(x) = kx$, Muzychuk \cite{Mu} proved the following
graph isomorphism between circulant graphs.

\begin{lemma}
\label{lem:iso}
If $n>1$ and $\gcd(k,n) = 1$, then $C_n(s_1,s_2,\ldots,s_t) \cong
C_n(ks_1,ks_2,\ldots,ks_t)$.
\end{lemma}

\noindent
As an example of Lemma \ref{lem:iso}, if
$n=7$ and $k=3$, then $C_7(1,3) \cong C_7(2,3)$. 

As noted in the introduction, the zero forcing number
(and minmal rank) of a number of families of graphs are known.  Some of these families (e.g., complete
graphs, cycles) are special cases of circulant graphs.  The next theorem summarizes some of these known results.

\begin{theorem}  Let $G = C_n(S)$ be a circulant graph.
\begin{enumerate}
    \item If $S = \{j\}$ and ${\rm gcd}(n,j) =1$,
    then $Z(G) = M(G) = 2$.
    \item If $S = \{1,\ldots,\lfloor \frac{n}{2} \rfloor\} \setminus 
    \{j\}$, ${\rm gcd}(n,j)=1$, and $n \geq 5$, then $Z(G) = M(G) = n-3$.
    \item If $S = \{1,2,\ldots, \lfloor \frac{n}{2}\rfloor \}$, then
    $Z(G) = M(G) = n-1$.
    \end{enumerate}
\end{theorem}

\begin{proof}
(1) The circulant graph 
$C_n(1) \cong C_n$, the $n$-cycle.  By
Lemma \ref{lem:iso}, $C_n(j) \cong C_n(1)$.
The result then follows from 
Lemma \ref{lem:facts}(3).

(2) The graph  $C_n(2,\ldots, \lfloor \frac{n}{2}\rfloor )$ is the complement of $C_n(1) \cong C_n$. If follows by Lemma \ref{lem:iso},  that if 
$S = \{1,\ldots,\lfloor \frac{n}{2} \rfloor\} \setminus\{j\}$
    and ${\rm gcd}(n,j) =1$, 
then $G \cong \overline{C_n}$.  Now
apply Lemma \ref{lem:facts}(4).

(3) Under this hypothesis, 
$G \cong K_n$; the conclusion follows from Lemma \ref{lem:facts}(2).
\end{proof}

Given $n \geq 1$ and $1 \leq d \leq \lfloor \frac{n}{2}
\rfloor$, the graphs $C_n(1,2,\ldots,d)$ are known as \emph{consecutive circulant graphs} (e.g., see
Figure \ref{fig:8}).
\begin{figure}[ht]
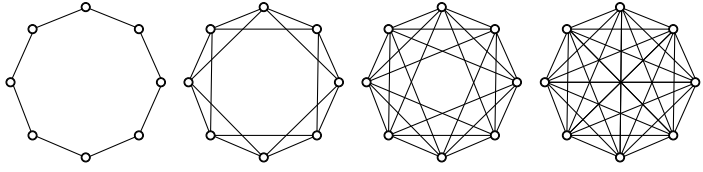

\[\Circulant{8}{1}\ \ \Circulant{8}{1,2}\ \ \Circulant{8}{1,2,3}\ \ \Circulant{8}{1,2,3,4}\]
\caption{The consecutive circulants $C_{8}(1)$, $C_{8}(1,2)$, $C_{8}(1,2,3)$, and $C_{8}(1,2,3,4)$.}
\label{fig:8}
\end{figure}
Deaett and Meyer determined the
zero forcing number and maximum nullity of
consecutive circulants.

\begin{theorem}\cite[Theorems 5.4 and 5.7]{DM}
\label{1234}
If $G =C_n(1,2,\ldots,d)$, then $M(G) = Z(G)=2d$.
\end{theorem} 

By combining the above result with Lemma \ref{lem:iso}, one can determine  $Z(G)$ and $M(G)$ for some other families of circulant graphs:

\begin{corollary}
Suppose $G = C_n(s,2s,3s,\ldots,ts)$ for $1 < ts < \frac{n}{2}$. 
\begin{enumerate}
\item If ${\rm gcd}(n,s)=s$, then $M(G) = Z(G)= 2st$.
\item If ${\rm gcd}(n,s)=1$, then $M(G) =Z(G) = 2t$.
\end{enumerate}
\end{corollary}

\begin{proof}
(1)  Suppose $\gcd(n,s)=s\neq 1$. By Theorem~\ref{thm:discon}, since ${\rm gcd}(n,s,2s,\ldots,ts) \neq 1$, the graph $G$ is disconnected. In particular, $G \cong sH$ with $H =C_{\frac{n}{s}}(1,2,3,\ldots,t)$.
 By Theorem~\ref{1234} $M(H) = Z(H)= 2t$, so the
 result follows from Lemma \ref{lem:facts}(5).

(2) By Lemma \ref{lem:iso}, 
$C_n(s,2s,\ldots,ts) \cong C_n(1,2,\ldots,t)$.  
Now apply Theorem \ref{1234}.
\end{proof}

\section{\textbf{Families of bipartite circulants}\label{bipartite}} 

In this section, we determine the
zero forcing number and minimal rank of 
some families of bipartite circulant graphs using
the work of Meyer \cite{M}.  Note that
Meyer investigates the family of bipartite graphs whose biadjacency
matrix is a circulant matrix.  These graphs are sometimes called {\it generalized bipartite circulants,}  although
in \cite{M}, for expediency, they are 
simply called bipartite circulants.  This usage
is different than our usage of the term bipartite
circulant graph.  More precisely, a {\it bipartite circulant} (as used in
this paper) is a circulant graph which is bipartite.  In particular, the family of bipartite circulants is a subclass of the generalized bipartite circulants,
the family of graphs studied in \cite{M}.  The reader should
be aware of the two usages when consulting \cite{M}.

Our starting point is the following characterization of
bipartite graphs due to Heuberger.

\begin{theorem}\cite[Theorem 1]{Heu}
\label{bicirc}
Let $G = C_{n}(s_1,\ldots,s_t)$ be a connected circulant. Then $G$ is bipartite if and only if $n$ is even and $s_1,\ldots,s_t$ are odd.
\end{theorem}
Note that partitioning the vertices of a bipartite circulant into parts based on the parity of their index will provide a bipartition of the vertex set. Two bipartite circulant
graphs are given in Figure \ref{8(13)}.
\begin{figure}[ht]
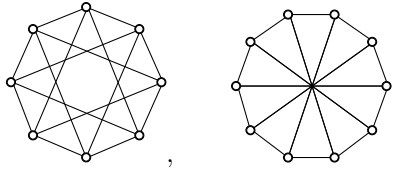

	\[\Circulant{8}{1,3},\qquad \Circulant{10}{1,5}\]
	\caption{The bipartite circulant graphs $C_{8}(1,3) \cong K_{4,4}$ and $C_{10}(1,5)$}
	\label{8(13)}
\end{figure}

Following Meyers \cite{M}, we can represent a bipartite
circulant graph using its biadjacency matrix.  Recall
that if $G$ is a biparitite graph with bipartition 
$V_1 \cup V_2$ and $m =|V_1|$ and $n = |V_2|$, then we can
represent $G$ by the $m \times n$ matrix $A$ where
$A_{ij}$ is $1$ if there is an edge between vertex
$v_i \in V_1$ and vertex $v_j \in V_2$, and 0 otherwise.  The matrix
$A$ is the {\it biadjacency matrix} of $G$. 

The next lemma describes how to represent
the biadjacency matrix of a bipartite circulant
graph.   Below, $P$ denotes the $n \times n$ permutation matrix corresponding
to the $n$ cycle
$(123\cdots n)$.  Note that $P^n = P^0 = I_n$ and
$P^a = P^b$ if $a \equiv b \pmod{n}$.

\begin{lemma} \label{biadj-form}
Let $G= C_{2n}(s_1,\ldots,s_t)$ be a bipartite circulant graph.
\begin{enumerate}
    \item[(1)]  If $s_t \neq n$, then the biadjacency
    matrix of $G$ is
    \[P^{\frac{s_1-1}{2}} + P^{\frac{s_2-1}{2}} +
    \cdots + P^{\frac{s_t-1}{2}} + P^{n-\frac{s_t+1}{2}} +
    \cdots +
    P^{n-\frac{s_2+1}{2}} + P^{n-\frac{s_1+1}{2}}. \]
    \item[(2)]  If $s_t = n$, then the biadjacency
    matrix of $G$ is
    \[P^{\frac{s_1-1}{2}} + P^{\frac{s_2-1}{2}} +
    \cdots + P^{\frac{s_t-1}{2} = n-\frac{s_t+1}{2}} +
    \cdots +
    P^{n-\frac{s_2+1}{2}} + P^{n-\frac{s_1+1}{2}}. \]
\end{enumerate}
\end{lemma}

\begin{proof}
This result is implicit in the proof of \cite[Theorem 2.2]{M}.
In particular, it is shown that if $s \in S$, then $s$ contributes
the matrices 
$P^{\frac{s-1}{2}}$ and $P^{-\frac{s+1}{2}}$ to the biadjacency matrix
of $C_{2n}(S)$ (there is a typo in \cite{M} where the author has an $n$ instead of an $s$).  Note that $P^{-\frac{s+1}{2}} = P^{n-\frac{s+1}{2}}$.
The result now follows by noting that if $s_t \neq n$, then
each $s_i \in S$ gives two distinct matrices, but when $s_t = n$,
the two matrices $P^{\frac{s_t-1}{2}}$ and $P^{-\frac{s_t+1}{2}}
= P^{n-\frac{s_t+1}{2}}$ are the same matrix.
\end{proof}

\begin{remark}
Note that in Lemma \ref{biadj-form}, the exponents of
the the matrices $P$ satisfy
\[\frac{s_1-1}{2} < \frac{s_2 -1}{2} < \cdots <
\frac{s_t-1}{2} \leq n-\frac{s_t+1}{2} < 
n - \frac{s_{t-1}+1}{2} < \cdots < n - \frac{s_1+1}{2}.\]
\end{remark}
We recall two further results from Meyer's paper \cite{M};  we 
have specialized his results to bipartite circulant graphs
of the form $C_{2n}(S)$.

\begin{lemma}\cite[Theorem 2.4]{M}
\label{bipartite-isom}
Suppose that $G = C_{2n}(S)$ is a connected bipartite circulant
graph with biadjacency matrix $P^{i_1} + \cdots + P^{i_r}$.
Then for each unit $a \in \mathbb{Z}/n\mathbb{Z}$ and
element $b \in \mathbb{Z}/n\mathbb{Z}$, the graph $G$
is isomorphic to the graph with biadjacency matrix
$P^{ai_1+b} + \cdots + P^{ai_r+b}$  where the exponents
are computed modulo $n$.
\end{lemma}

\begin{theorem}\cite[Corollary 3.5]{M}
\label{bipartite-sequential}
Suppose that $G=C_{2n}(S)$ is a connected bipartite
circulant graph with biadjacency matrix $P^0+P^1 + P^2
+ \cdots + P^t$.  Then $M(G) = Z(G) = 2t$.
\end{theorem}

We now come to the main results of this section.

\begin{theorem} \label{bipartiteresult1}
Fix $n,\ell \in \mathbb{N}$ with
$\ell \geq 1$ and $n \geq 2\ell+2$.
\begin{enumerate}
    \item[(1)] If $n$ is odd and 
    $G  =C_{2n}(n-2\ell,n-2\ell+2,\ldots,n-2,n)$, then
    $M(G) = Z(G) = 4\ell$.
    
    \item[(2)] If $n$ is even and $G = C_{2n}(n-2\ell-1,
    n-2\ell+1,\ldots,n-3,n-1)$, then $M(G) = Z(G) = 4\ell+2$.
\end{enumerate}
\end{theorem}

\begin{proof}
(1) Set $k= \frac{n-1}{2}$.  Then by 
Lemma \ref{biadj-form} (2), the biadjacency
matrix of $G$ has the form
\[P^{k-\ell} + P^{k -\ell+1} 
+ \cdots + P^{k} + P^{k+1} + \cdots + P^{k+\ell}.\]
Let $b= k-\ell$.  Then by Lemma \ref{bipartite-isom},
$G$ is isomorphic to the graph with biadjacency matrix
\[P^{k-\ell-b =0} + P^1 + P^2+ \cdots + P^{k+\ell-b}.\]
Since $k+\ell-b = 2\ell$, by Theorem \ref{bipartite-sequential} we get 
$M(G) = Z(G) = 4\ell$.

The proof of (2) is similar.  Let $k =\frac{n-2}{2}$. 
By Lemma \ref{biadj-form}, the biadjacency matrix of $G$ has 
the form
\[P^{k-\ell} + P^{k-\ell+1} + \cdots + P^{k} + 
P^{k+1} + \cdots + P^{k+\ell+1}.\]
Using Theorem \ref{bipartite-isom}, this graph
is isomorphic to the graph with biadjacency matrix
\[P^{0}+P^{1}+\cdots +P^{k+\ell+1-k+\ell}.\]
So, Theorem \ref{bipartite-sequential} gives us
the conclusion $M(G) = Z(G) = 2(2\ell+1) =4\ell+2$.
\end{proof}

For our last result, we require the following
result about complete bipartite graphs.

\begin{theorem}\label{biparitteresult2}
Fix $n,\ell \in \mathbb{N}$ with $n > 1$,
$\ell \geq 1$, and $n \geq 2\ell-1$.
\begin{enumerate}
    \item  If $2\ell-1 \leq n-1$ and 
    $G = C_{2n}(1,3,\ldots,2\ell-1)$, then 
    $M(G) = Z(G) = 4\ell-2$.
    \item If $2\ell-1=n$ and 
    $G=C_{2n}(1,3,\ldots,2\ell-1)$, then
    $M(G) = Z(G) =4\ell-4$.
\end{enumerate}
\end{theorem}

\begin{proof}   
	The proof of both statements are similar to the 
	proof of Theorem \ref{bipartiteresult1}.
	
	(1)  By Lemma
	\ref{biadj-form}, the biadjacency matrix of
	$G$ has the form
	\[P^0+P^1+\cdots +P^{\ell-1}+P^{n-\ell}+\cdots+
	P^{n-1}.\]
	Adding $\ell$ to each exponent, by Lemma
	\ref{bipartite-isom}, the graph $G$ is isomorphic
	to the graph with the biadjacency matrix
	\[P^0+P^1+ \cdots + P^{\ell-1+\ell}.\]
	Employing Theorem \ref{bipartite-sequential} 
	gives us $M(G) = Z(G) = 2(2\ell-1) =4\ell-2$.
	
	(2)  If $G = C_{2n}(1,3,\ldots,n)$, then 
	the biadjacency matrix is $P^0 + P^1 + \cdots + P^{n-1}$,
	and consequently, $M(G) = Z(G) = 2n-2 = 2(2\ell-1)-2
	=4\ell-4$ by Theorem \ref{bipartite-sequential}.
	Alternatively, one notes that $G = K_{n,n}$, and
	so the result follows, for example, by
	\cite[Observation 3]{BHL}.
	\end{proof}

\section{\textbf{Circulants which are torus products}\label{torus}}

\begin{figure}
	\label{fig:k2mob}
	\begin{center}
		\begin{tikzpicture}
		\vertexdef
		
		\node (n0) at (0,-0.5) {$x_{2,1}$};
		\node (n1) at (0,1.5)  {$x_{1,1}$};
		\node (n2) at (1,-0.5) {$x_{2,2}$};
		\node (n3) at (1,1.5) {$x_{1,2}$};
		\node (n4) at (2,-0.5)  {$x_{2,3}$};
		\node (n5) at (2,1.5) {$x_{1,3}$};
		\node (n6) at (3,-0.5)  {$x_{2,4}$};
		\node (n7) at (3,1.5) {$x_{1,4}$};
		\node (n8) at (4,-0.5) {$x_{2,5}$};
		\node (n9) at (4,1.5)  {$x_{1,5}$};
		\node (n10) at (5,-0.5) {$x_{2,6}$};
		\node (n11) at (5,1.5)  {$x_{1,6}$};

		[thick,scale=0.8]%

		\foreach \x in {0}
		\draw (0,0) -- (5,0);
		\draw (0,1) -- (5,1);

		\draw (0,0) -- (5,1);
		\draw (0,1) -- (5,0);

		\draw (0,0) -- (0,1);
		\draw (1,0) -- (1,1);
		\draw (2,0) -- (2,1);
		\draw (3,0) -- (3,1);
		\draw (4,0) -- (4,1);
		\draw (5,0) -- (5,1);

		\node[cv] (0,0) {};
		\node[cv] at (1,0) {};
		\node[vertex] at (2,0) {};
		\node[vertex] at  (3,0) {};
		\node[vertex] at  (4,0) {};
		\node[vertex] at  (5,0) {};
		\node[cv] at (0,1) {};
		\node[cv] at (1,1) {};
		\node[vertex] at  (2,1) {};
		\node[vertex] at  (3,1) {};
		\node[vertex] at  (4,1) {};
		\node[vertex] at  (5,1) {};
		\end{tikzpicture} 
		\qquad
		\begin{tikzpicture}
		\vertexdef
		
		\node (n0) at (0,-0.5) {$v_0$};
		\node (n1) at (0,1.5)  {$v_6$};
		\node (n2) at (1,-0.5) {$v_1$};
		\node (n3) at (1,1.5) {$v_7$};
		\node (n4) at (2,-0.5)  {$v_2$};
		\node (n5) at (2,1.5) {$v_8$};
		\node (n6) at (3,-0.5)  {$v_3$};
		\node (n7) at (3,1.5) {$v_{9}$};
		\node (n8) at (4,-0.5) {$v_4$};
		\node (n9) at (4,1.5)  {$v_{10}$};
		\node (n10) at (5,-0.5) {$v_{5}$};
		\node (n11) at (5,1.5)  {$v_{11}$};

		[thick,scale=0.8]%

		\foreach \x in {0}
		\draw (0,0) -- (5,0);
		\draw (0,1) -- (5,1);

		\draw (0,0) -- (5,1);
		\draw (0,1) -- (5,0);

		\draw (0,0) -- (0,1);
		\draw (1,0) -- (1,1);
		\draw (2,0) -- (2,1);
		\draw (3,0) -- (3,1);
		\draw (4,0) -- (4,1);
		\draw (5,0) -- (5,1);

		\node[cv] (0,0) {};
		\node[cv] at (1,0) {};
		\node[vertex] at (2,0) {};
		\node[vertex] at  (3,0) {};
		\node[vertex] at  (4,0) {};
		\node[vertex] at  (5,0) {};
		\node[cv] at (0,1) {};
		\node[cv] at (1,1) {};
		\node[vertex] at  (2,1) {};
		\node[vertex] at  (3,1) {};
		\node[vertex] at  (4,1) {};
		\node[vertex] at  (5,1) {};

		\end{tikzpicture}\end{center}
	\caption{The M\"{o}bius ladder $K_2 \mob  C_6 \cong C_{12}(1,6)$ with a zero forcing set.}
	\label{mobiusladder}
\end{figure}
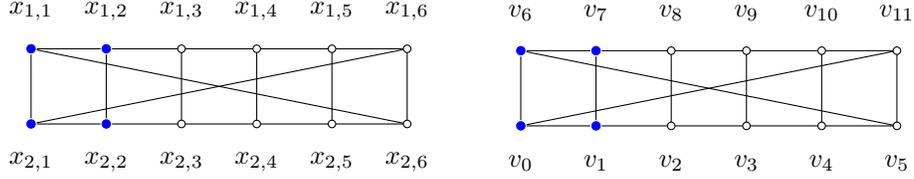

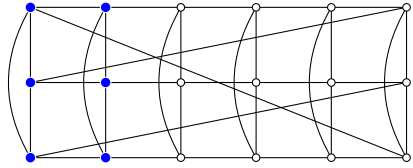
\begin{figure}
	\label{fig:k3mob}
	\begin{center}
		\begin{tikzpicture}
		\vertexdef
		
		[thick,scale=0.8]%

		\foreach \x in {0}
		\draw (0,0) -- (5,0);
		\draw (0,1) -- (5,1);
		\draw (0,2) -- (5,2);

		\draw (0,0) -- (5,1);
		\draw (0,1) -- (5,2);
		\draw (0,2) -- (5,0);

		\draw (0,0) -- (0,2);
		\draw (1,0) -- (1,2);
		\draw (2,0) -- (2,2);
		\draw (3,0) -- (3,2);
		\draw (4,0) -- (4,2);
		\draw (5,0) -- (5,2);
		
		\draw (0,0) to[out=120, in=240] (0,2);
		\draw (1,0) to[out=120, in=240] (1,2);
		\draw (2,0) to[out=120, in=240] (2,2);
		\draw (3,0) to[out=120, in=240] (3,2);
		\draw (4,0) to[out=120, in=240] (4,2);
		\draw (5,0) to[out=120, in=240] (5,2);

		\node[cv] (0,0) {};
		\node[cv] at (1,0) {};
		\node[vertex] at (2,0) {};
		\node[vertex] at  (3,0) {};
		\node[vertex] at  (4,0) {};
		\node[vertex] at  (5,0) {};
		\node[cv] at  (0,1) {};
		\node[cv] at  (1,1) {};
		\node[vertex] at  (2,1) {};
		\node[vertex] at  (3,1) {};
		\node[vertex] at  (4,1) {};
		\node[vertex] at  (5,1) {};
		\node[cv] at  (0,2) {};
		\node[cv] at  (1,2) {};
		\node[vertex] at  (2,2) {};
		\node[vertex] at  (3,2) {};
		\node[vertex] at  (4,2) {};
		\node[vertex] at  (5,2) {};
		\end{tikzpicture}
	
		\end{center}
	\caption{The torus product $K_3 \mob  C_6 \cong C_{18}(1,6)$ with a zero forcing set.}\label{secondqm}
\end{figure}

In this section we extend the definition of the 
M\"obius ladder to a type of torus product.
The zero forcing number for the M\"{o}bius ladder was calculated in \cite{AIM} to be four (e.g., see Lemma \ref{lem:facts}(7)).  We compute the zero forcing number for
our torus products, and as a corollary, we are able to compute
the zero forcing number for a new family of circulant graphs.  We
also give evidence for a conjecture on the minimal rank of this family.

Recall that the Cartesian product of the graphs $G$
and $H$ with 
$V(G)=\{x_1,x_2,\ldots,x_n\}$
and 
$V(H)=\{y_1,y_2, \ldots,y_m\}$
is the graph $G\sq H$ with 
vertex set
$V=\{(x,y) ~|~x\in V(G), y\in V(H)\}$ with two vertices $(x_i,y_j)$, $(x_k,y_\ell)$ adjacent if 
either $i=k$ and $y_j$ is adjacent to $y_{\ell}$ in $H$, or $j=\ell$ and $x_i$ is adjacent to $x_k$ in $G$.  We position the vertices
of $G \sq H$ in a $n \times m$ grid  such that  the $i$-th column
contains the vertices $(x_k,y_i)$, 
for $1\leq k\leq n$,
and the $j$-th
row contains the vertices $(x_j,y_k)$, for $1\leq k\leq m$.
Then $G \sq H$
essentially  consists of 
$m$ copies of $G$ as columns and $n$ copies of 
$H$ as rows. The product
$C_n \sq C_m$ can be pictured as a lattice on a  
torus (see \cite{MMNR}).
For $m\geq 3$, define the \emph{torus product}
graph $G \mob C_m$ to consist of $m$ copies, $G_1,\ldots G_m$, of 
$G$ with $G_i$ having vertices 
$x_{1,i},x_{2,i},\ldots,x_{n,i}$ with edges between copies as follows: for $1\leq i\leq m-1$
and $1\leq k\leq n$, 
$x_{k,i}$ is adjacent to $x_{k,i+1}$
and, with subscript addition modulo $n$, $x_{i,m}$ is adjacent to $x_{i+1,1}$.  
Then the M\"{o}bius ladder is 
simply the torus product $K_{2}\mob C_m$ (see for example Figures ~\ref{mobiusladder} and \ref{secondqm}).
Note that the torus product
$C_n\mob C_m$ is referred to
as a twisted torus in 
\cite{MMNR}. 

The following proof takes advantage of the
fact that the torus product $K_n \mob C_m$ is 
locally similar to the Cartesian product $K_n \sq  C_m$. 
Let $G=K_n\sq C_m$ for $m\geq 4$. It was shown in \cite{AIM} that
$Z(G)=M(G)=2n$. Below we give an alternate argument that $Z(G)=2n$; the same argument
applies to the torus 
product $K_n \mob C_m$. 

\begin{theorem}\label{thm:mobK}
Let $G=K_n \sq C_m$ or $G=K_n \mob C_m$. If $m\geq 4$, then  $Z(G)= 2n$. If $m=3$ and $n\geq 3$, then
$Z(G)=2n-1$. If $m=3$ and $n=2$, then 
$M(G)=Z(G)=4$. 
\end{theorem}

\begin{proof}
Let $G=K_n \sq C_m$ or $G=K_n \mob C_m$.

First consider the case $m\geq 4$.  Assume the vertices are in a grid as described
before the theorem.
The argument uses the fact that locally, about a column of vertices, the
graphs of $K_n \sq C_m$ and $K_n \mob C_m$ both have the subgraph structure $K_n \sq P_3$. (In fact there is an automorphism of $G$ that takes column $G_i$ to $G_k$  for any $i,k$.)

Observe that if $2n$ vertices of two adjacent
copies of $K_n$ are filled, this set is a zero forcing set of $G$. Thus, $|Z(G)|\leq 2n$.

Let $F$ be a minimal zero forcing set for $G$. 
Pick a forcing vertex $v\in F$. All but one neighbour of $v$ must
be in $F$. Since $G$ is $(n+1)$--regular, $|F|\geq n+1$. 
Once a force is made from $v$, the filled vertices are all those vertices
in some copy of $K_n$ (a column), plus two additional vertices in some copy of $C_m$ (a row). 
These $n+2$ filled vertices are in three consecutive columns, say $L,M,$ and $R$, the middle column $M$ being completely filled. 

In order for a vertex outside of these three columns to force some other vertex,
it must be in a column that already has $n$ filled vertices.
If these $n$ vertices were originally in $F$, then $|F|\geq (n+1)+(n)=2n+1$. Thus, before a vertex outside of the three columns can do any forcing, some vertex in one of the three columns must first force a vertex outside these columns.  

Without loss of generality, let $u$ be the first vertex in column $R$ used to force a vertex outside
the three columns. Then the remaining vertices in $R$ are already filled. Suppose $r\geq 0$ vertices of column $R$ are in $F$. Then the remaining
$n-r$ vertices in column $R$ must have been forced from vertices in column $M$. For any vertex in column
$M$ to force a vertex in column $R$, there must already be a vertex in column $L$, in the same row,
that is filled. This implies that there must be at least $n-r-1$ vertices in 
column $L$ that are in $F$. Hence $|F|\geq (n+1)+r+(n-r-1)=2n$. 
Therefore $Z(G)=2n$.

Now consider the case $m=3$. Let $F$ be a forcing set. As noted above, for a vertex  $v\in F$ to force another vertex, there must be at least $(n-2)$ other vertices of $F$ in the same column as $v$. As labelled above, $M$ must start off with
either (a) $n-1$ vertices in $F$ or (b) $n$ vertices in $F$. Note that in case (a), there must be at least one row with vertices of $F$ in both $L$ and $R$. In either case,  after a force is made, the set of all filled vertices must then contain all the vertices of $M$ and two vertices, one in $L$ and one in $R$, adjacent to a common vertex in $M$. If the zero forcing set $F$ contains
$\ell$ vertices of $L$ and $r$ vertices of $R$, then the vertices of $M$ could force at most $\ell$ vertices of $R$ and $r$ vertices of $L$. Thus, after forcing, at most $\ell+r$ vertices of $L$ (and their corresponding vertices in $R$ are filled). If $\ell+r <n-1$, then no further forcing can occur. Thus $\ell+r\geq n-1$.
In fact, if $t$ is the number of rows that have a vertex of $F$ in both $L$ and $R$, then $\ell+r -t \geq n-1$. 
In case (a), $t\geq 1$ and in case (b) $t\geq 0$.
It follows that $|F|\geq 2n-1.$ 
To construct a minimal zero forcing set, let $F$ consist of the
first $n-1$ vertices of $L$ (with $L$ as column 1), the first $n-1$ vertices of $M$ (with $M$ as column 2), and
the first vertex of $R$ (column 3). Then we claim
that $F$ is  a forcing set with exactly $2n-1$ vertices. In particular $x_{2,1}$ can force $x_{n,1}$ and then $x_{1,2}$ can force $x_{n,2}$. From here the remaining vertices in $R$ can be forced by
vertices in $M$. 
Thus, $Z(G)=2n-1$.
 
Finally, for the case $m=3$ and $n=2$,  $G$ is a M\"obius ladder and
so by \cite{AIM}, $M(G)=Z(G)=4$. 
\end{proof}

\begin{figure}[ht]
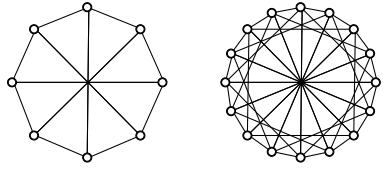

\[\Circulant{8}{1,4} \qquad \Circulant{16}{1,4,8}\]
\caption{The circulant graphs $C_{8}(1,4)\cong K_2 \mob C_4$
and $C_{16}(1,4,8)\cong K_4 \mob C_4$.}
\label{fig:mob}
\end{figure}

\begin{remark}
As seen above, the proof takes advantage of the shared local structure of the 
Cartesian and torus  products. Note the only difference between
the two graphs is the particular permutation of adjacencies between the
first column $G_1$ and the last column $G_n.$ As such, the same result holds
true for a much larger class of graphs, if all of the $n!$ permutations of the adjacencies
between the first and last column are considered, not just the two specified by $G\sq C_m$ and 
$G\mob C_m$.
\end{remark}

Theorem \ref{thm:mobK} can now be applied to 
the study of circulant graphs.
In the next theorem  $m\geq 2$ since if $m=1$ then the graph considered is a consecutive circulant which is already discussed in Theorem~\ref{1234}.

\begin{theorem}\label{thm:mobC} 
Fix $n \geq 3$, and 
let $G = C_{nm}(1,m,2m,\ldots,bm)$ with 
$b = \floor{\frac{n}{2}}$.
If $m=2$, then $Z(G)=n+1$.
If $m=3$, then $Z(G)=2n-1$. If $m\geq 4$, then $Z(G)=2n$.
\end{theorem}

\begin{proof}
Let $G = C_{nm}(1,m,2m,\ldots,bm)$ with $b = \floor{\frac{n}{2}}$.
Then 
$G$ can be obtained by combining the edges of $C_{nm}(1)\cong C_{nm}$ and
$C_{nm}(m,2m,\ldots bm) \cong mC_n(1,2,\ldots,b)=mK_n$.
One can then observe that for $m\geq 3$,  $C_{nm}(1,m,2m,\ldots, bm)$ is the 
torus 
product $K_n \mob C_m$ (see for example, 
Figures ~\ref{secondqm} and \ref{fig:mob}). The zero forcing number can then be obtained from Theorem~\ref{thm:mobK}. 

If $m=2$, then $G$ is a $(n+1)$--regular graph and by 
Lemma~\ref{lem:facts}(\ref{lem:kregular}),
$Z(G)\geq n+1$. Taking any vertex $v$ and all but one of its neighbours provides a zero forcing set of size $n+1$. In particular, $v$ will force its only unfilled neighbour. 
The remaining  $n-2$ unfilled vertices can be forced consecutively from the neighbours of $v$ with subscripts that have the same parity as that of $v$. 
\end{proof}

We expect that $M(G)=Z(G)$ for all the graphs in Theorem~\ref{thm:mobC}. 
 Using special matrices,
 Theorem~\ref{thm:mrankmobius}
and Theorem~\ref{thm:KnC6}  demonstrate that $M(G)=Z(G)$ for these graphs when $m=4$ and $m=6$. 

An $n\times n$ \emph{circulant Hankel} matrix $H$ is a matrix for which each row is 
shifted one position to the left from the row above it with a wrap around to the end
of the row. In particular, if the first row of $H$ is $(a_1,a_2,\ldots, a_n)$, then
the $k$th row of $H$ is $(a_k,a_{k+1},\ldots, a_{k-1})$.  For example, 
$$H = 
\left[ \begin{array}{rrrr}
1&2&4&-2\\   
2&4&-2&1\\
4&-2&1&2\\
-2&1&2&4
\end{array}\right]
$$
is a circulant Hankel.  Note that the reverse diagonals of a Hankel matrix are
constant and consequently the matrix is symmetric. 
Let $$P=\left[
\begin{array}{ccccc}
0&\cdots&\cdots&0&1\\
1&\ddots&&&0\\
0&\ddots&\ddots&&\vdots\\
\vdots&\ddots&\ddots&\ddots&0\\
0&\cdots&0&1&0\end{array}
\right].
$$ 
Note that if $H$ is a circulant Hankel matrix, then $PH=HP^T$ and $PH$ is itself a circulant Hankel matrix. 

\begin{lemma}\label{lem:Hank}
For $n\geq 3$, there exists an orthogonal circulant Hankel matrix $A$ such that both $A$ and $A-PA$ have no zero
entries.  	
	\end{lemma}

\begin{proof}
Let $H$ be a circulant Hankel matrix with first row $\mathbf{a}=(1,2,2^2,\ldots,2^{n-2},w)$ and $w=-\frac{2}{3}\left(2^{n-2}-1\right)$.
We claim that $H$ has orthogonal rows. Since $H$ is circulant, it is enough to show that the first row of $H$ is orthogonal
to every other row of $H$. If $\mathbf{b}$ is row $(k+1)$
of $H$, $1\leq k\leq n-1$, then
$\mathbf{b}=(2^k,2^{k+1},\ldots, 2^{n-2},w,1,2,\ldots,2^{k-1})$ and 
\begin{eqnarray*}
\mathbf{a}\mathbf{b}^T&=&\sum_{i=0}^{n-k-2}2^{k+2i} + 2^{n-k-1}w+\sum_{i=0}^{k-2}2^{n-k+2i}+2^{k-1}w \\ 
 &=& 2^k\left(\frac{4^{n-k-1}-1}{3}\right)+w2^{k-1}\left[2^{n-2k}+1\right]+2^{n-k}\left(\frac{4^{k-1}-1}{3}\right)\\
 &=&\frac{1}{3}\left[2^{2n-k-2}-2^k+2^{n+k-2}-2^{n-k}\right] + w2^{k-1}\left[2^{n-2k}+1\right]\\
 &=&\frac{1}{3}\left[2^{n-k}+2^k\right]\left[2^{n-2}-1\right] +w2^{k-1}\left[2^{n-2k}+1\right]=0.
\end{eqnarray*}
Thus $H^2=\lambda I$ with $\lambda=||\mathbf{a}||^2=w^2+\sum_{i=0}^{n-2}2^{2i}$.
Therefore $A=\frac{1}{\sqrt{\lambda}} H$ is an orthogonal circulant Hankel matrix with no zero entries. 
The fact that $A-PA$ has no zero entries follows from the fact that 
 $\mathbf{a}_i\neq\mathbf{a}_{i+1}$ for $1 \leq i < n$
 where ${\bf a}_i$ denotes 
 row $i$ 
 of $H$.
\end{proof}

\begin{theorem}\label{thm:mrankmobius}
	Given $n\geq 3$, $M(K_n\mob C_4) = 2n$.
\end{theorem}

\begin{proof}
	Let $A$ be an $n \times n$ orthogonal circulant Hankel matrix as in Lemma~\ref{lem:Hank}. 
Let $$K=\left[ 
\begin{array}{cccc}
A&I&O&P^T\\
I&B&I&O\\
O&I&-PA&I\\
P&O&I&-PB
\end{array}
\right]$$	
with $B=A-PA$. Note that $B$ is a circulant Hankel matrix and  $AB=I-P^T$ since $PA=AP^T$. 	
Also, $BA=I-P$ since $A$ and $B$ are symmetric. Let $$E=
\left[
\begin{array}{cccc}
I&-A&-P^T&O\\
O&I&O&O\\
O&O&I&O\\
O&-P&PB&I
\end{array}
\right].$$
Then, using the fact that $AB=I-P^T$ and $BA=I-P$, 
$$EK=\left[
\begin{array}{cccc}
O&O&O&O\\
I&B&I&O\\
O&I&-PA&I\\
O&O&O&O
\end{array}
\right].$$ Since $E$ is invertible, it follows that nulllity$(K)\geq 2n$.	
Note that $K$ is a symmetric matrix with 
graph $K_n\mob C_4$, since $A$, $B$, $PA$ and $PB$ are symmetric matrices with no zero entries. Therefore, $M(K_n\mob C_4)\geq 2n$ and thus by Theorem~\ref{thm:mobK}, $M(K_n\mob C_4)=2n$. 
	\end{proof}

Since $K_n\mob C_4$ is the circulant $C_{4n}(1,4,8,\ldots,4b)$ with $b = \floor{\frac{n}{2}}$,
we have the following:

\begin{corollary}\label{m=4}
If $G = C_{4n}(1,4,8,\ldots,4b)$ with $b = \floor{\frac{n}{2}}$, then $M(G)=Z(G)=2n$.		\end{corollary}

\begin{theorem}\label{thm:KnC6}
	Given $n\geq 3$, $M(K_n\mob C_6) = 2n$.
	\end{theorem}

\begin{proof}
Let $A$ be an $n \times n$ orthogonal circulant Hankel matrix as in Lemma~\ref{lem:Hank}. Let
$$K=\left[ 
\begin{array}{cccccc}
A&I&0&0&0&P^T\\
I&A&I&0&0&0\\
0&I&A&I&0&0\\
0&0&I&PA&I&0\\
0&0&0&I&PA&I\\
P&0&0&0&I&PA
\end{array}
\right].	
$$	
The graph of $K$ is $K_n\mob C_6$. If
$$E=\left[
\begin{array}{rrrrrr}
I&-A&0&A&-P^T&0\\
0&I&0&0&0&0\\
0&0&I&0&0&0\\
0&0&0&I&0&0\\
0&0&0&0&I&0\\
0&-P&PA&0&-PA&I
\end{array}
\right],
$$	
then 
$$EK=\left[ 
\begin{array}{cccccc}
0&0&0&0&0&0\\
I&A&I&0&0&0\\
0&I&A&I&0&0\\
0&0&I&PA&I&0\\
0&0&0&I&PA&I\\
0&0&0&0&0&0
\end{array}\right],
$$ 
noting that $APA=AAP^T=AA^TP^T=P^T.$
Since that first and last $n$ rows of $EK$ are zero, and $E$ is invertible, it follows that $M(G)\geq 2n$ and by Theorem~\ref{thm:mobK}, $M(G)=2n$.
\end{proof}

\begin{corollary}\label{m=6}
If $G=C_{6n}(1,6,12,\ldots,6\floor{\frac{n}{2}})$ 
then $M(G)=Z(G)=2n$.
\end{corollary}

\begin{remark}
Based upon Corollaries \ref{m=4} and \ref{m=6}, we wonder if 
%
$G = C_{mn}(1,m,2m,\ldots,m\floor{\frac{n}{2}})$ implies
$M(G)=Z(G)$
in general.
\end{remark}
There is another circulant graph that is isomorphic to a torus 
product.
In particular, the graph $C_{nm}(1,m)\cong C_n\mob C_m$. (Note that 
for $m\neq n$, $C_m\mob C_n$ is not isomorphic to $C_n\mob C_m$.) 
 
 The following theorem mimics a result \cite{Betal} on the
 zero forcing number of the Cartesian product $C_n \sq C_m$, which is not
 surprising since locally,  the graph is the same as $C_n\mob C_m$.  In particular,
 it was shown in \cite{Betal} that, with $n\geq m\geq 3$, then $M(C_m\sq C_m)=Z(C_m\sq C_m)= 2m-1$ if $m$ is odd,
 and otherwise $M(C_n\sq C_m)=Z(C_n\sq C_m)= 2m$.
 
 We do not know if $M(G)=Z(G)$ in general for $G=C_n\mob C_m$,
 but the zero forcing number is bounded above in same was as the 
 Cartesian product, except when $m=n$ and  $m$ is even.
 In this case, $Z(C_m\mob C_m)<Z(C_m\sq C_m)$. The argument
 is similar to that in \cite[Theorem 2.18]{deAlbaetal}.
 
 \begin{theorem}\label{mobCC}
 	Suppose $n\geq 3$ and $m\geq 3$. Then
 	$Z(C_n\mob C_m) \leq  
 	\begin{cases} 
 	2 \min\left\{ n,m\right\} &  \mbox{if} ~~m\neq n \\
 	2m-1 &  \mbox{if} ~~m=n. 
 	\end{cases}$
\end{theorem}

\begin{proof} Let $G=C_n\mob C_m$. First note that if two consecutive columns of vertices of  
	$G$ are in a set $F$, then $F$ is a zero forcing set of
	 $G$. In particular, each of these two columns can force all the vertices on a neighbouring
	 column. Likewise, if two consecutive rows of $G$ are in some set $F$, then $F$
	 is a zero forcing set. In particular, suppose two consecutive rows of vertices of $G$ are in $F$. 
	 By symmetry, one can assume the first two rows of $G$ are in $F$.
	 In this case, one can force the whole third row, left to right: in particular, 
	 $x_{2,j}$ can force $x_{3,j}$ as $j$ ranges from 1 to $m$. (Note that 
	 $x_{2,m}$ cannot force $x_{3,m}$ until $x_{3,1}$ has been filled.)
	 Consequently, each of the subsequent rows can also be forced. 
	 Therefore, $Z(G)\leq 2 \min\left\{ m,n\right\}.$
	 
	 Now suppose $m=n$. Let $k=\lceil \frac{m}{2} \rceil$. 
	 Suppose $F$ consists of the $m$ vertices of column $k$, and
	 $m-1$ vertices of column $k+1$, namely $\{ x_{2,(k+1)},\ldots,x_{m,(k+1)}\}$.  
	 Then column $k$ can force $m-1$ vertices in column $k-1$, in rows 2 through $m$. 
	 Now the two columns with $m-1$ vertices can each force
	 $m-3$ vertices in the columns $k+2$ and $k-2$ respectively, namely in rows 3 though $m-1$. This can
	 be repeated, filling two less vertices in each column until
	 the end columns are reached. 
	 If $m$ is odd, then the forcing above will result in the two vertices
	 in each of columns 1 and $m$ being filled. In particular, rows $k$ and $k+1$
	 will be completely filled, and so the resulting set will force the remaining
	 rows to be filled, as noted at the beginning of the proof. 
	 If $m$ is even, then column 1 will have the three vertices $x_{k,1}, x_{(k+1),1},$ and $x_{(k+2),1}$ filled
	 but column $m$ will only have vertex $x_{(k+1),m}$ filled. However, $x_{(k+1),1}$ can force $x_{k,m}$.
	 At this point, rows $k$ and $k+1$ are completely filled, and so the remainder of the
	 vertices can be forced. Therefore, $Z(C_m\mob C_m)\leq 2m-1$.
	\end{proof}

	\begin{corollary}\label{cor:Cmn}
		If $n,m\geq 3$,  then $Z(C_{m^2}(1,m))\leq 2m-1$,  and if $n\neq m$, then 
		$Z(C_{nm}(1,t))\leq 2 \min\left\{m, n\right\} $
	for each $t\in \{m,n\}.$		
		\end{corollary}
	
	Corollary~\ref{cor:Cmn} deals with the case that
	$n\geq 3.$ (An earlier version of the $C_{m^2}(1,m)$ case is found in \cite{LD}.) For the case with $n=2$, see 
	Theorem~\ref{2n(1,n)}.
			Determining $M(G)$ for the graphs $G$ in Corollary~\ref{cor:Cmn} does not seem to be straightforward. For the case
	with $m=n=3$, the following theorem demonstrates that  $M(G)=Z(G)$.

	\begin{theorem}\label{Z(G)=5}
		If $G=C_{9}(1,3)$, then $M(G)=Z(G)=5$.
				\end{theorem} 

\begin{proof} Let $G=C_9(1,3)$. 
	By Corollary~\ref{cor:Cmn}, $Z(G)\leq 5$. By Theorem~\ref{Mbound}, it is enough to 
	show that $M(G)\geq 5.$ Let 
	$$A=\left[ \arraycolsep=1.4pt\def\arraystretch{1.4}
	\begin{tiny}
	\begin{array}{rrr|rrr|rrr} 
	-\frac{1}{8}&\frac{3}{4}&-\frac{1}{2}&1&0&0&0&1&0\\
	\frac{3}{4}&-2&\frac{1}{2}&0&1&0&0&0&1\\
	-\frac{1}{2}&\frac{1}{2}&-\frac{3}{4}&0&0&1&1&0&0\\ \hline 
	1&0&0&\frac{48}{5}&-\frac{12}{5}&-\frac{24}{5}&-\frac{16}{5}&0&0\\
	0&1&0&-\frac{12}{5}&\frac{6}{5}&\frac{4}{5}&0&\frac{12}{5}&0\\
	0&0&1&-\frac{24}{5}&\frac{4}{5}&\frac{4}{5}&0&0&-\frac{2}{5}\\ \hline
	0&0&1&-\frac{16}{5}&0&0&-\frac{2}{5}&-\frac{4}{5}&-\frac{3}{5}\\
	1&0&0&0&\frac{12}{5}&0&-\frac{4}{5}&\frac{24}{5}&\frac{6}{5}\\
	0&1&0&0&0&-\frac{2}{5}&-\frac{3}{5}&\frac{6}{5}&-\frac{3}{10}
	\end{array}\end{tiny}
	\right].$$
Then, $A\in \mathcal{S}(G)$ and rank($A)=4$ and, therefore, $M(G)\geq 5$.
	\end{proof}

	For the family of circulant graphs of Corollary~\ref{cor:Cmn}, we can also find a lower bound
	using the following result based on the girth of a graph. The \emph{girth} of a graph $G$ is size of the smallest cycle
	in $G$.
	
	\begin{theorem}\cite{DKS} \label{girththeorem}
	Let $G$ a graph with girth $g \geq 3$ and minimal
	degree $\delta \geq 2$.  Then $$Z(G) \geq 
	(g-3)(\delta-2)+\delta.$$  
	\end{theorem}
	
	The above theorem was first conjectured in \cite{DK}, and
	proved in some special cases.  The proof of Theorem \ref{girththeorem} was completed in \cite{DKS} (and see the references within this paper).  This result is now applied to $C_{mn}(1,t)$.
	
	\begin{theorem}\label{lowerbounds}
	Let $t\in\{n,m\}$ with $n,m\geq 3.$ If $G = C_{nm}(1,t)$, then 
	$Z(G) \geq \begin{cases}
	4 & \mbox{if }  3t=nm \\
	6 & \mbox{otherwise.}
	\end{cases}$
	\end{theorem}
	
	\begin{proof} Let $G = C_{nm}(1,t)$.
	The graph $G$ contains a four cycle: $\{v_0,v_1,v_{t+1},v_t\}.$
	Since, $t\geq 3$,
	in order for $G$ to have a cycle of length three, then $3t=nm$. In particular, $G$ would contain
	the $3$-cycle  $\{v_0,v_t,v_{2t}\}$. 
	The theorem then follows from  Theorem~\ref{girththeorem} since $G$ is 
	4-regular.
	\end{proof}
	
	In the case $n=3$ the zero forcing number equals the lower bound given in Theorem \ref{lowerbounds}.

	\begin{theorem}If $G=C_{3m}(1,3)$, then
	$Z(G) = \begin{cases} 
 	6 & \mbox{if  } m>3 \\
 	5 & \mbox{if  } m=3. 
 	\end{cases}$
	\end{theorem}
	
	\begin{proof}
	When $m=3$, the result follows from Theorem \ref{Z(G)=5}.  When $m > 3$, the upper bound
	of Corollary \ref{cor:Cmn} and the lower
	bound of Theorem \ref{lowerbounds}  agree.
	\end{proof}

In this section we considered graphs $C_{nm}(1,t)$ with $n,m\geq 3.$ In the next section we consider the case $m=2$. 

\section{\textbf{The cubic circulant graphs}\label{cubic}} 

A circulant graph is {\it cubic} if it is three regular. 
If $G$ is a cubic circulant graph, 
then  $G = C_{2m}(a,m)$ for some $a$ with $1 \leq a < m$.  The cubic circulant graphs were characterized in \cite{DD}.

\begin{theorem}
\label{gcd=t}
Let $G = C_{2m}(a,m)$ with $1 \leq a < m$, and let $t = gcd(a,2m)$.
\begin{enumerate}
  \item If $\frac{2m}{t}$ is even, then $G= 
  tC_{\frac{2m}{t}}\left(1,\frac{m}{t}\right)$.
  \item If $\frac{2m}{t}$ is odd, then $G = 
  \frac{t}{2}C_{\frac{4m}{t}}\left(2,\frac{2m}{t}\right)$.
\end{enumerate}
\end{theorem}

	Theorem \ref{gcd=t} demonstrates that  
	the zero forcing number of a cubic circulant is going to be 
	a multiple of the zero forcing number of 
    a circulant of the form $C_{2m}(1,m)$ or $C_{2m}(2,m)$.	
	 The  zero forcing number is calculated for these classes in Theorems \ref{2n(1,n)} and \ref{2n(2,n)} respectively.

\begin{theorem}
\label{2n(1,n)}
Suppose $m \geq 2$, and $G = C_{2m}(1,m)$.
Then $M(G)=Z(G)=  
 \begin{cases} 
       3 &  \mbox{if} ~~m=2 \\
       4 &  \mbox{if} ~~m\geq 3. 
   \end{cases}
$
\end{theorem}

\begin{proof} 
	Let $G = C_{2m}(1,m)$. If $m =2$, then $G = K_4$ and, hence, $M(G)=Z(G)=3$ by Lemma~\ref{lem:facts}(2).
	Suppose $m\geq 3$.	Then $G$ is isomorphic to the M\"obius ladder $K_2 \mob C_m$, and hence $M(G)=Z(G)=4$ as noted
	in \cite{AIM}.
\end{proof}

\begin{theorem}
\label{2n(2,n)} Suppose $m\geq 3$ and $G=C_{2m}(2,m)$.
\begin{enumerate}
  \item If $m$ is odd, then 
  $M(G)=Z(G)=\min\{ m, 4\}$.
  \item If $m$ is even, then 
  $M(G)=Z(G)=
  \begin{cases} 
      6 & \mbox{if}~~m=4 \\
      8 & \mbox{if}~~m\geq 6.
   \end{cases}
  $
 \end{enumerate}
\end{theorem}

\begin{proof} Let $G=C_{2m}(2,m)$. Suppose $m=2k$ for some positive integer $k$.  Then $\gcd(2m,2,m)=2$, and
	so $G\cong 2C_{2k}(1,k)$.  Thus, $M(G)=Z(G)$ by Theorem~\ref{2n(1,n)}.

	Suppose $m$ is odd. 
	 Note that $C_{2m}(2)$ is a subgraph of $G$ consisting of two disjoint cycles of length $m$. 
	 Observe that the cycle containing vertex $v_0$ consists of the vertices with even subscripts. The other cycle will consist of
	 the vertices with odd subscripts. For any vertex $v_i \in V(G)$, its neighbours will be $\{v_{i+m}, v_{i+2},v_{i-2}\}$. 
	 Thus, besides the edges of the two aforementioned cycles, the graph $G$ also contains the perfect
	 matching consisting of edges of the form $\{v_i,v_{i+m}\}$.
	  	  It follows that $C_{2m}(2,m) \cong C_m \sq  K_2$. The result 
	  	  follows from Lemma~\ref{lem:facts}(\ref{lem:KsC}).
\end{proof}

\begin{theorem}\label{thm:an}
Let $G = C_{2m}(a,m)$ with $1 \leq a < m$, and let $t = gcd(a,2m)$.
\begin{enumerate}
  \item If $\frac{2m}{t}$ is even, then 
  $M(G)=Z(G) = 
  \begin{cases} 
      3t & \mbox{if}~~ m=2t \\
      4t & \mbox{if}~~ m\geq 3t.
   \end{cases}
  $
    \item If $\frac{2m}{t}$ is odd, then $M(G)=Z(G) = 
    \begin{cases} 
      m & \mbox{if}~~ 2m=3t \\
      2t & \mbox{otherwise}.
   \end{cases}
  $
\end{enumerate}
\end{theorem}

\begin{proof} 
The theorem follows from Theorem~\ref{gcd=t} and
Theorems~\ref{2n(1,n)} and \ref{2n(2,n)},
along with Lemma \ref{lem:facts}(5).
\end{proof}

Note that when $a$ and $m$ are odd, then the cubic circulant graphs
$C_{2m}(a,m)$ in Theorem~\ref{thm:an} are 
further examples of bipartite circulants 
discussed in Section~\ref{bipartite}.

\section{\textbf{Concluding comment}}\label{conclusion}

For every circulant graph $G$ for which we have calculated $M(G)$ and $Z(G)$, these two numbers have been equal. Equality also holds for the extreme cases; when $G=K_n$ or $G=C_n$. We wonder if equality holds 
for every circulant graph in general.

\medskip 





{\bf Acknowledgments} 
To draw the circulant graphs in this paper,
we used the \LaTeX\ script of Eastman \cite{E}.
Some of the results in this paper
first 
appeared 
the MSc thesis of Riddell \cite{riddell}.
This research was supported in part
by grants held by the last two authors, 
NSERC RGPIN-2016-03867 and 
NSERC RGPIN-2019-05412, respectively.

\end{document}